  \documentclass[12pt, twoside]{article}
  \usepackage{amsmath,amsthm,amssymb}
  \usepackage{times}
  \usepackage{enumerate}
 \usepackage{diagrams}  
  \theoremstyle{definition}
  \newtheorem{thm}{Theorem}[section]
  \newtheorem{cor}[thm]{Corollary}
  \newtheorem{proposition}[thm]{Proposition}
  \newtheorem{lem}[thm]{Lemma}
  \newtheorem{defin}[thm]{Definition}
  \newtheorem{rem}[thm]{Remark}
  \newtheorem{rems}[thm]{Remarks}

  \numberwithin{equation}{section}

 \usepackage{epsfig} 

 \usepackage[english]{babel}

\usepackage{fancyhdr}
\usepackage{cite}
\usepackage{graphicx,color}
\usepackage{enumitem}

\usepackage{srcltx}

\setenumerate[0]{label=(\alph*)}

\usepackage{hyperref}
\hypersetup{%
    colorlinks=True, 
    citecolor=blue,
    linkcolor=black}

\newcommand{\beqn}{\begin{eqnarray*}}
\newcommand{\eeqn}{\end{eqnarray*}}

\newcommand{\ben}{\begin{equation}}
\newcommand{\een}{\end{equation}}

\newcommand{\beq}{\begin{eqnarray}}
\newcommand{\eeq}{\end{eqnarray}}

\newcommand{\benn}{\begin{equation*}}
\newcommand{\eenn}{\end{equation*}}

\newcommand{\VV}{  X }

\DeclareMathOperator{\dt}{\frac{d}{dt} }

\DeclareMathOperator{\im}{im }

\newcommand{\R}{{\mathbf{R}}}
\newcommand{\Hs}{{\mathbb{H}}}
\newcommand{\N}{{\mathbf{N}}}

\newcommand{\fr}{\textnormal{fr}}

\newcommand{\befr}{\begin{framed}}
\newcommand{\efr}{\end{framed}}

\begin{document}

\title{A structure theorem for shape functions defined on submanifolds}
\author{Kevin Sturm}
\date{}

\maketitle




\begin{abstract}
In this paper, we study shape functions depending on closed 
submanifolds.
We prove a new structure theorem that establishes the general structure of the shape derivative 
for this type of shape function. 
As a special case we obtain the 
classical Hadamard-Zol\'esio structure theorem, but also 
the structure theorem for cracked sets can be recast into our framework. 
As an 
application we investigate several unconstrained shape functions arising from 
differential geometry and fracture mechanics.
\end{abstract}



\section{Introduction}
The classical structure theorem \cite{MR2731611,Delfour1992}  for real valued  shape functions plays a crucial role in 
shape optimization both from the numerical and the theoretical point of view. 
Given a shape function $J$, the structure theorem  states that the 
 shape derivative $X\mapsto dJ(\Omega)(X)$ at an open or closed  set $\Omega$ has support in the boundary 
$\partial \Omega$. This is a consequence of Nagumo's invariance theorem for 
ordinary differential equation. If the boundary of $\Omega$ 
is additionally of class $C^{k+1}$, $k\ge 0$, and  $X\mapsto dJ(\Omega)(X)$ is  linear and continuous,
then it can be shown that there is a linear and continuous function $g:C^k(\partial \Omega) \rightarrow \R$
such that
\ben\label{structure_1}
dJ(\Omega)(X) = g(X|_{\partial \Omega}\cdot \nu),
\een
where 
$\nu$ a normal vector field  along $\partial \Omega$.

In \cite{MR1803575} the structure theorem 
was extended to subsets $\Omega$ of the plane that  have a (smooth) 
fissure/crack of codimension one. A smoothly cracked set $\Omega$ in the plane
is a smooth set  $\Omega$ 
from which we remove the image $\Sigma := \gamma([0,1])$ of an embedded 
$C^{k+1}$ curve $\gamma:[0,1]\rightarrow \R^2$. In other words 
$\Omega := \tilde \Omega \setminus \Sigma$. 
The set $\tilde \Omega$ is no longer of class $C^{k+1}$
and hence the classical structure theorem does not apply.  
However, it can be shown that in this case the structure of the shape 
derivative is 
\ben\label{structure_2}
dJ(\Omega)(X) =  h(X|_{\Sigma}\cdot \mathfrak n) +  a \gamma'(0)\cdot X_{\gamma(0)} + b \gamma'(1)\cdot X_{\gamma(1)} ,
\een
where $h:C^k(\Sigma) \rightarrow \R^2$ is linear and continuous and 
$a,b$ are two real numbers and $\mathfrak n$ denotes the normal vector field 
along $\Sigma$.  In  \cite{laurain:tel-00139595} this 
theorem was extended to sets $\Omega\subset \R^d$, but still with a crack 
of codimension one.

Recently, in \cite[p. 3 Theorem 1.3]{LamPierr06} it was shown that 
if $\Omega$ has merely finite perimeter, then it is still possible to obtain a 
formula like \eqref{structure_1}. But it is clear  that in this case a normal 
vector field is not readily available anymore.  
However, one can use one of the  generalisations of the normal vector field from  geometric measure theory.
Then the structure theorem reads
\ben\label{structure_3}
 dJ(\Omega)(X) =  \mathfrak g(\VV|_{\Gamma^*}\cdot \nu_*)
\een
where $\nu_*$ is the generalized normal and $\Gamma^*:=\partial^*\Omega$ denotes the reduced boundary of $\partial \Omega$.
The cracks and corners are hidden in the notion of generalised normal and the 
function $\mathfrak g$ is  defined on a bigger space than $C^k(\Gamma)$. 

In this paper, we prove a structure theorem for shape functions defined on closed submanifolds of $\R^d$ with 
or without boundary. 
As a first side  product we are now able to  extend the structure theorem of  
\cite{laurain:tel-00139595} to arbitrary codimensions of cracks. 
A second striking consequence is that our
new structure theorem gives the structure of many other
functionals occuring in differential geometry. 
 The proof is very different from the one given in  
\cite{laurain:tel-00139595} and thus also contributes in giving a new 
perspective on the subject. 

In Section~\ref{sec:1}, we briefly recall some facts about  submanifolds with 
boundary and introduce shape functions and the Eulerian semi-derivative. 

In Section~\ref{sec:2}, we give a detailed reinterpretation of  Nagumo's 
invariance condition for the case of submanifolds.
This version requires some notions from differential geometry.

In Section~\ref{sec:3}, we are going to revisit the structure theorem for smooth domains and give a slightly different proof, than 
what is known in the literature as this will be useful for our further study. 

In Section~\ref{sec:4}, the main result is proved by first studying a general 
splitting of vector fields on submanifolds.

In Section~\ref{sec:5}, we are presenting several examples. 
 
\section{Preleminaries}\label{sec:1}
\subsection{Submanifolds of $\R^d$ with boundary}
We begin with the definition of a submanifold $M$ of $\R^d$, $d\ge 1$.  Let us denote the open half space in $\R^d$ by 
$$\mathbb H^d:=\{x\in \R^d| x=(x_1,\ldots, x_d),\, x_d>0\}.$$
The boundary of the half space $\partial \mathbb H^d= \R^{d-1}\times \{0\}$
is identified with $\R^{d-1}$. When $U$ is an open subset of  
$\bar\Hs^d := \overline{\Hs^d} = \R^{d-1}\times [0,\infty)$, then we define 
its interior and boundary as
$\text{int}(U ) := U \cap  \mathbb H^d$ and 
$\partial U := U \cap \partial \mathbb H^d$, respectively. Note that the  
boundary $\partial U$ does not 
coincide with the topological boundary of $U$. 

\begin{defin}
Let $1\le m \le d$. A subset $N$ of $\R^d$ is called {\bf $n$-dimensional $C^k$- 
submanifold of} $\R^d$, $k\ge 1$, if 
for each point $p$ in $N$ there is an open set $U$ of $\R^d$ containing $p$, 
an open set $V$ of $\R^d$, and a $C^k$-diffeomorphism $\varphi:U \rightarrow V$,  such that
$$ \varphi(U\cap N)  = V \cap (\R^m\times \{{\bf 0}\}).$$
Here, ${\bf 0 } = (0,\ldots, 0)^\top\in \R^{d-m}. $ The tuple $(U\cap N,\varphi)$ is called 
{\bf chart} and $\varphi^{-1}$ is called {\bf parametrisation}.
\end{defin}


\begin{defin}[Submanifolds with boundary]
A subset $M$ of a $n$-dimensional $C^k$-submanifold $N$ is called $m$-{\bf dimensional 
$C^k$-submanifold of $N$ with boundary} if for every $p$ in $M$ there is a chart $(U,\varphi)$ 
of $N$ around $p$ , such that
$$ \varphi(U\cap M)  = \varphi(U) \cap (\bar\Hs^m\times \{{\bf 0}\}) \subset \R^n.$$
Here, $p$ is called {\bf boundary point} if $\psi(p)$ lies in $ \partial  \Hs^m := \partial  \Hs^m \times \{{\bf 0}\}$. The set of boundary points is denoted by 
$\partial M$ and we define the {\bf interior} of $M$ by $\text{int}(M):= M\setminus \partial M$. 
In order to avoid any confusion we are going to denote by $\partial M$ the boundary in the above sense and by $\fr(M)$ ($\fr$ = frontier) the topological boundary of the set $M$.
\end{defin}

\begin{rems}
\begin{itemize}
\item[(a)] The boundary $\partial M$ is a $m-1$-dimensional submanifold without boundary, that is, $\partial(\partial (M)) = \emptyset$.
 The interior $\text{int}(M)$ is a $m$-dimensional 
submanifold without boundary.
\item[(b)] Note that the image $\varphi(U\cap M)$ is not open in $\R^d$, but 
relatively open.
\item[(c)] Open subsets $U$ of $\R^d$ are $C^\infty$-submanifolds {\bf without} boundary.
           (Note that they do have a topological boundary.)
\item[(d)] One and two dimensional submanifolds of $\R^d$ are called {\bf 
	embedded curve} and {\bf embedded surface}, respectively.  Analogously  $d-1$ dimensional 
	submanifolds of $\R^d$ are 
{\bf embedded hypersurfaces}.
\item[(f)] It is always possible to replace the open set $U$ by another open 
	set $\tilde U$  in such a way that
$\varphi(\tilde U)$ is the unit ball $\mathbb B^d$ in $\R^d$ centered at the origin. 
\end{itemize}
\end{rems}

We introduce the tangent space 
at a point $p$ of $M$
by $ T_pM := d_{\varphi(p)}(\varphi^{-1})(\R^m) $ and similarly the tangent space of $\partial M$ at a point $p$ 
is given by $ T_p(\partial M) := d_{\varphi(p)}(\varphi^{-1})(\R^{m-1}). $
This can also  be expressed in a different way by ($q=\varphi(p)$)
\begin{align*} 
T_pM & = \text{span}\{ \partial_{x_1}\varphi^{-1}(q), \ldots,  \partial_{x_m}\varphi^{-1}(q)\}\\
T_p(\partial M) & = \text{span}\{ \partial_{x_1}\varphi^{-1}(q), \ldots,  \partial_{x_{m-1}}\varphi^{-1}(q)\}. 
\end{align*}
Here, $\R^m\subset \R^d$ has to be understand as the 
image of the natural injection 
$ x \mapsto (x, 0) \in \R^d $.
Setting $T_p^{\pm}M := d_{\varphi(p)}(\varphi^{-1})(\pm \bar\Hs^m)$, we have 
for $p\in M$ that $T_pM = T_p^{+}M\cup T_p^{-}M$ and $T_p(\partial M) = T_p^{+}M\cap T_p^{-}M$.
We call the disjoint collection $TM := \cup_{p\in M} T_pM$ of tangent spaces 
also tangent bundle of  $M$.  The tangent bundle is a smooth $2m$-dimensional manifold if $M$ is smooth. 
Similarly, $T(\partial M)$ denotes the $2(m-1)$-dimensional tangent bundle at $\partial M$. 
Note that for $p\in \partial M$ the tangent space $T_p(\partial M)$ is a $m-1$-dimensional subspace of the $m$-dimensional vector space $T_pM$. As such $T_p(\partial M)$ is also an inner product space with Euclidean scalar product of $\R^d$. Consequently there are exactly two 
unit vectors 
$\pm\nu(p)$ in $T_pM$ that are normal  to $T_p(\partial M)$. We call $\nu$ 
outward-pointing unit vector field if for all $p\in \partial M$,
$\nu(p)\in T^+_pM$; cf. \cite{AmmEschIII}. 
In the sequel, we always denote the outward-pointing unit normal field  by $\nu$.  Its  uniqueness and existence
along $\partial M$ is guaranteed by  \cite[p. 346 
Prop. 13.26]{Lee}.

\subsection{Eulerian semi-derivative}
Let $X:\R^d\rightarrow \R^d$ be a vector field 
satisfying a global Lipschitz condition: there is a constant $L>0$ 
such that 
$$ |X(x)-X(y)| \le L|x-y| \quad \text{ for all } x,y\in \R^d .$$ 
Then  we 
associate with $X$ the flow $\Phi_t$ by solving for all $x\in \R^d$
$$ \frac{d}{dt} \Phi_t(x) = X(\Phi_t(x)) \;\text{ on } [-\tau,\tau], \quad \Phi_0(x)=x.$$
The global existence of the flow is ensured by the theorem of  Picard-Lindel\"of and hence 
$\Phi:\R\times \R^d \rightarrow \R^d$.

Subsequently, we restrict ourselves to a special class  of vector fields, namely 
$C^k$-vector fields with compact support in some fixed set. 
To be more precise for a fixed  open set $D\subset \R^d$, we consider vector 
fields belonging to  $C^k_c(D,\R^d)$.
We equip the space $C^k_c(D,\R^d)$ respectively $C^\infty_c(D,\R^d)$ with the topology induced 
by the following  family of semi-norms: for each compact $K\subset D$ and muli-index 
$\alpha\in \N^d$ with $|\alpha| \le k$ we define
$ \|f\|_{K,\alpha} := \sup_{x\in K} |\partial^\alpha f(x) |. $
With this familiy of norms the space $C^k_c(D,\R^d)$ becomes a locally convex 
vector space.  

Next, we recall the definition of the Eulerian semi-derivative.

\begin{defin}\label{def1} 
Let $D\subset \R^d$ be an open set.  Let $J:\Xi \rightarrow \R$ be a shape function defined on a set $\Xi$ of subsets
of $D$ and fix  $k\ge 1$. Let $\Omega\in \Xi$ and $X \in C^k_c(D,\R^d)$ be such that 
$\Phi_t(\Omega) \in \Xi$ for all 
$t >0$ sufficiently small.   
Then the  Eulerian semi-derivative of $J$ at $\Omega$ in direction $X$ is defined by
\ben
dJ(\Omega)(X):= \lim_{t \searrow 0}\frac{J(\Phi_t(\Omega))-J(\Omega)}{t} .
\een
\begin{itemize}
\item[(i)] The function $J$ is said to be \textit{shape differentiable} at $\Omega$ if  $dJ(\Omega)(X)$ exists for all $X \in C^\infty_c(D,\R^d)$ and 
$ X     \mapsto dJ(\Omega)(X) $ 
is linear and continuous on $C^\infty_c(D,\R^d)$.
\item[(ii)]  The smallest integer  $k\ge 0$ for which $X \mapsto dJ(\Omega)(X)$ is continuous with respect to the $C^k_c(D,\R^d)$-topology is called the order of $dJ(\Omega)$. 
\end{itemize}
\end{defin}
The set $D$ in the previous definition is usually called hold-all domain or 
hold-all set or universe.

\subsection{Quotient space}\label{sec:quotient_space}
 Henceforth, for all structure theorems to be considered,  we define for an arbitrary set $A\subset D$ the linear space
\ben\label{def:vector_zero}
 T^k(A):= \{X\in  C^k_c(D,\R^d)| \, X = 0 \text{ on } A \}. 
 \een
By definition $T^k(A) \subset C^k_c(D,\R^d)$ and $T^k(A)$ is closed. 
 We introduce an equivalence relation on 
$ C^k_c( D,\R^d)$ by
\ben\label{eq:equiv_rela}
 X\sim Y, \quad  X,Y\in  C^k_c(D,\R^d) \quad \Leftrightarrow \quad X=Y \; \text{ on } A 
 \een 
and denote the set of equivalence classes and its elements by 
$\mathcal Q^k(A) := C^k_c( D,\R^d)/T^k(A)$
 and $[\VV]$, respectively.

We denote by $\mathfrak J_A$ the restriction mapping of vector field 
belonging to $C^k_c(D,\R^d)$ to mappings $A\rightarrow \R^d$, that is,  
$$\mathfrak J_A: C^k_c(D,\R^d) \rightarrow A^{\R^d},\quad  X\mapsto X_{|A},$$
where $A^{\R^d}$ denotes the space of all mappings from $A$ into $\R^d$.
The mapping $\mathfrak J_A$ induces the mapping 
$\tilde{\mathfrak{J}}_A:Q^k(A) \rightarrow \R$ as
depicted in Figure \ref{diagram-1}. 
Hence by definition $\mathfrak J_A =  \tilde{\mathfrak{J}}_A\circ \pi  $. 

\begin{figure}
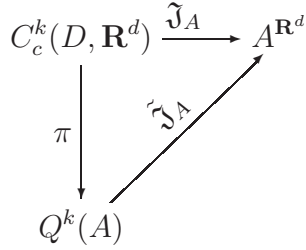
\label{diagram-1}
\begin{diagram}
	C^k_c(D,\R^d) & \rTo^{\mathfrak J_A}  & A^{\R^d} \\
	\dTo^{\pi} &      \ruTo^{\tilde{\mathfrak{J}}_A}                  &  \\
 Q^k(A)       &                           & 
\end{diagram}
\caption{Restriction mapping $\mathfrak J_A$ and induced mapping 
	$\tilde{\mathfrak{J}}_A$. }
\end{figure}

 The semi-norms on $C^k_c(D,\R^d)$ induce 
 semi-norms on the  quotient space
 $$ \|[\VV]\|_{K,\alpha}:= \inf_{\tilde X\in T^k(A)} \|X-\tilde X\|_{K,\alpha} = \inf_{\xi\in C^k_c( D,\R^d)}\{\|\xi\|_{K,\alpha}:\, \xi = \VV \text{ on } A \}.  $$
Let $f:C^k_c(D,\R^d) \rightarrow \R$ be a linear function respecting the equivalence 
relation \eqref{eq:equiv_rela}, that is, if $X\sim Y$ then it follows 
$f(X)=f(Y)$. Then  $f:C^k_c(D,\R^d) \rightarrow \R$ is continuous if and 
only if its induced function $\tilde f: Q^k(A) \rightarrow \R $ is continuous. 
So if $f$ is continuous, then for every compact $K\subset \R^d$ there is a constant $C>0$ 
such that for all  $\psi\in T^k(A)$ and all multi-indicies $\alpha$ with 
$|\alpha| \le k$ we have 
$$ |\tilde f([X])| =   |f(X)|=|f(\VV-\psi)| \le C \|\VV-\psi\|_{K,\alpha} $$
 and hence $ |\tilde f([X])| \le C\|[X]\|_{K,\alpha}$. 
 Later we will see that the shape derivative $dJ(\Omega)$ in an open or closed 
 set $\Omega\subset \R^d$ respects the above equivalence relation.  This 
 will follow from Nagumo's theorem considered in the next section. 

\section{Nagumo's theorem}\label{sec:2}

\subsection{Nagumo's invariance condition}
Nagumo's theorem states roughly the following: if a given vector field defined
on some closed subset of $\R^d$ is tangent to that set at each point,
then the solutions of the 
associated ordinary differential equation cannot 
leave this closed set. 

In order to make this tangency requirement precise, we define for a given
subset $K\subset \R^d$ the  Bouligand contingent cone 
to $K$ at $x\in \overline K$:
$$
T_K(x):=\{v\in \R^d|\;v =  \lim_{n\rightarrow \infty}(x_n-x)/t_n\text{ for some } x_n\rightarrow_D x,\; t_n\searrow 0\}.
$$
Here $x_n\rightarrow_D x$ indicates that $x_n\in D$ and 
$x_n \rightarrow x$ as $n\rightarrow \infty$.
The following result is Nagumo's classical theorem \cite[Theorem 2, p. 180]{AubCell}; cf. also  \cite{MR0015180, MR2731611}.
\begin{thm}\label{thm:nagumo}
	Let $K$ be a closed subset of a Hilbert space $H$ and $f$ a
	continuous function from $K$ into $H$ satisfying the tangential
	condition
	$ \forall x\in K, \; f(x) \in T_K(x). $
	Then for each $x_0\in K$, there exists $T>0$ such that
	the ODE $x'(t) = f(x(t)), x(0)=x_0$ has a viable trajectory on
	$[0,T]$. 
\end{thm}
By ``viable solution'' we  means that $x(t)\in K$ for all $t\in [0,T]$. 
\begin{cor}\label{cor:nagumo}
Let $K\subset \R^d$ be a closed set and $X:K\rightarrow \R^d$ a vector field 
satisfying a global Lipschitz condition. Assume that for all  $x\in K$ we have $\pm X(x)\in T_K(x).$
Then the flow $\Phi_t$ of $X$ is for each 
	$t$ in $[-\tau,\tau]$ a bijection from $K$ onto $K$. In particular, 
	$\Phi_t(K) =K $ for all $ t\in \R .$
\end{cor}
\begin{proof}
	By Kirszbraun's theorem (cf. \cite{kirszbraun,valentine1,valentine2}) we may extend the vector field $X:K\rightarrow \R^d$ 
	to a globally Lipschitz continuous  vector field 
	$\tilde X:\R^d\rightarrow \R^d$ having the same Lipschtiz  constant and 
	satisfying 
	$X=\tilde X$ on $K$. The Picard-Lindel\"of theorem ensures that the flow $\tilde \Phi_t$ is 
	globally defined, that is, $\tilde \Phi: \R \times \R^d \rightarrow \R^d$.  
	Applying Theorem \ref{thm:nagumo} to $K$ yields $\Phi_t(K) \subset K$ for all $t$ in $[0,\infty)$. On the other hand 
	we also have $\Phi_{-t}(K) \subset K$
	for all $t$ in $[0,\infty)$ as $-X(x)\in T_K(x)$ for all $x$ in $K$. Together 
	we obtain  $ K = \Phi_t(\Phi_{-t} ( K )) \subset 
	\Phi_t(K) $ for all $t\in \R$ and thus $\Phi_t(K) = K$.	
\end{proof}

\begin{cor}\label{cor:nagumo-1}
Let $D\subset \R^d$ be an open set with $C^k$-boundary, $k\ge 1$.  
Suppose that  $X:\R^d\rightarrow \R^d$
is a vector field satisfying a global Lipschitz condition and  $X\cdot \nu =0$ 
on $\fr(D)$.  Then $\Phi_t(D)=D$ and $\Phi_t(\fr(D)) = \fr(D)$ for all 
$t$ in $[-\tau, \tau]$. 
\end{cor}
\begin{proof}
We have for all 
$x\in \fr(D)$ the inclusion $T_xD\subset T_D(x)$.  For 
all interior points $x\in D$ it is easily checked that $T_D(x) = \R^d$. 
So the assumptions of Corollary \ref{cor:nagumo} are satisfied.  
Since $D$ is open we have $ \overline D = \partial D \cup D$. Moreover, since 
	$\partial D$ is closed and $\pm X(x) \in T_{\partial D} (x) $ for all 
	$x\in \partial D$, we also have $\Phi_t(\partial D) = \partial D
	$ and it follows that $\Phi_t(D) = D$. 
\end{proof}

\subsection{Nagumo's theorem for submanifolds}
In this section we give a proof of the following version of Nagumo's theorem needed for the 
further analysis. 

\begin{proposition}\label{prop:nagumo}
	Let $M$ be a closed $m$-dimensional $C^k$-submanifold of $\R^d$, $k\ge 1$. 
	Suppose we are given a vector field  $X:\R^d \rightarrow \R^d$ 
	of class $C^1$ with compact support satisfying
	\begin{align}
	X_p  & \in T_p(\text{int}(M)), \quad \text{ for all } p \in T_p(\text{int}(M)), \label{con:1}\\
	X_p  & \in T_p (\partial M), \quad \text{ for all } p \in \partial M. \label{con:2}
	\end{align}
	Then the flow $\Phi_t=\Phi_t^X$ of $X$ is a $C^k$-diffeomorphism 
	$\Phi_t: M\rightarrow M$ and thus in particular
	\begin{align} 
	\Phi_t(\text{int}(M)) & =\text{int}(M) \quad \text{ for all } t \\
	\Phi_t(\partial M) & = \partial M \quad \text{ for all } t. 
	\end{align}
\end{proposition}
\begin{proof}
	We first show that for each $p$ in $ M$ 
	and each curve $\alpha$ solving 
	$$ \alpha'(t) = X(\alpha(t)) \; \text{ in } [-\tau,\tau], \; \alpha(0)=p \quad \Longrightarrow \quad  \alpha(t)\in  M \text{ for all } t \text{ in } [-\tau,\tau].$$
	For each $p$ in $\partial M$, there exist an open neighboorhood $U$ of 
	$p$ in $\R^d$, an open set $V$ in $\R^d$ 
	and  $C^k$-diffeomorphism $\varphi:U\rightarrow V$, such that 
	$ \varphi(U\cap \partial M ) = V \cap (\R^{m-1}\times\{ {\bf 0}\} ). $
  Let $\alpha$ solve 
	$ \alpha'(t) = X(\alpha(t)) $, $\alpha(0) = p$, and define 
	$\tilde \alpha(t):= \varphi(\alpha(t))$ and
	$$
	\tilde \alpha(t) := \varphi(\alpha(t)), \quad  \tilde X(y)  :=\varphi^{-*} 
	( \partial \varphi X)(y). $$
	Then we compute 
	\ben\label{eq:nagumo_locally}
	\begin{split}		
		\tilde \alpha ' (t)  & = \partial \varphi(\alpha(t)) \alpha'(t) \\
		& =   \partial \varphi(\alpha(t)) X(\alpha(t))\\
		& = \tilde X(\tilde \alpha(t))
	\end{split}
	\een
	for all $t$. 
	We have that 
	$\{v_{i,p} := d \varphi^{-1}_{\varphi(p)}(e_i)\}$, $i=1,\ldots, m-1$ is a basis of 
	$T_p(\partial M)$ and thus we may write locally 
	$$ X_p = \sum_{i=1}^{m-1} \alpha_{i,p} v_{i,p} \quad \Rightarrow \quad \tilde X\circ \varphi =  d_p\varphi( X) = \sum_{i=1}^{m-1} \alpha_{i,p} e_i,    $$
	where $\{e_1,\ldots, e_{m-1}\}$ denotes the canonical basis of $\R^{m-1}$. 
	But his means that the last $d-m+1$ components of $\tilde{X}$ are zero.  In view of  \eqref{eq:nagumo_locally} we obtain
	$  \tilde \alpha_{m,p}' = \cdots = \tilde \alpha_{d,p}' =0 $ and 
	taking into account the initial condition we get  
	$ \tilde \alpha_{m,p}= \cdots = \tilde \alpha_{d,p} =0$.  
	Define the subinterval $\mathcal T: = \{t: \alpha(t) \in U\cap \partial M\}$ of $[-\tau,\tau]$.  We obtain 
	$$      \tilde \alpha(t) \in \varphi(U)\cap (\R^{m-1}\times \{{\bf 0} \})\quad  \text{ for all } t\in \mathcal T, 	$$
	which is equivalent to 
	$  \alpha(t) \in \partial M$ for all $ t\in \mathcal T.$ 
	This shows that the curve stays on the boundary of  $M$ as long we are 
	in the chart $U$. However, if we enter another chart, we can proceed the 
	same argumentation as above and obtain $\mathcal T = [-\tau,\tau]$.  
	In a similar way, 
	we may show that if $p\in \text{int}(M)$ and 
	$\alpha:[-\tau,\tau] \rightarrow M$ is a  $C^1$-curve, 
	such that $p=\alpha(0)$, then $\alpha([-\tau,\tau]) \subset \text{int}(M)$. 
	It follows $\Phi_t(\partial M)\subset \partial M$ and 
	$\Phi_{-t}(\partial M) \subset\partial M$, which implies  
	$ \partial M = \Phi_t(\Phi_{-t}( \partial M)) \subset \Phi_t(\partial M)  $
	and thus $\Phi_t(\partial M) = \partial M$ for all $t$.  In the same we obtain 
	$\Phi_t(\text{int}(M)) = \text{int}(M)$ for all $t$. Now the rest of the statement 
	is clear.  
\end{proof}

\begin{rem}
	The invariance $\Phi_t(M) = M$ can also be proved by directly using
	Theorem \ref{thm:nagumo}. It can be shown that for all 
	$p\in \text{int}(M)$
	$$ \pm X_p \in T_{\text{int(M})} (x) \quad \Leftrightarrow \quad  X_p\in T_p(\text{int}(M)) $$
	and for all $p\in \partial M$
	$$ \pm X_p \in T_{M}(x) \quad \Leftrightarrow \quad \pm X_p\in T_p^+ M  \quad \Leftrightarrow \quad  X_p\in T^+_pM \cap T^-_pM = T_p(\partial M). $$
	So in fact the conditions 
	\eqref{con:1}-\eqref{con:2} are reformulations of: for all $
	 p\in M, \pm X_p \in T_M(x).
	$ 
	However, in order to give a self contained presentation we  gave a direct proof. 
\end{rem}
\begin{rems}
	\begin{enumerate}
		\item The hypothesis that $M$ be (relatively) closed can not be dropped as can be seen by 
		considering open subsets $\Omega\subset \R^d$ as submanifolds equipped with the 
		identity chart. In this case the conclusion of the previous proposition does not hold. 
		\item Note that the conditions \eqref{con:1}-\eqref{con:2} state that the map 
		$p\mapsto X_p$ defines a vector field (smooth section of the tangent bundle) $ X:M\rightarrow TM $ such that its restriction 
		to $\partial M$ is also a vector field on $\partial M$, that is, 
		$ X:\partial M\rightarrow T(\partial M).  $
		Note that this is not the case in general.
		\item  The case $m=d$ corresponds to the case of the closure of an open set  $M$ in  $\R^d$ with $C^k$-boundary 
		$\partial M = \fr(M) $. The conditions 
		\eqref{con:1}-\eqref{con:2} reduce to 
		$ X_p\cdot \nu_p = 0$ for all $p\in \partial M, $
		where $\nu$ is the inward pointing normal vector along $\partial M$.         
		Indeed, for all $p\in \text{int}(M)$, we have $X_p\in T_pM = 
		\R^d$ which is always true and for all $p\in \partial M$, we have 
		$ X_p \in T_p(\partial M)$ if and only if $X_p \cdot \nu_p = 0$.  
		\item Let $M$ be simply connected.  In the  case  $m=1$ the 
			manifold $M$ can be described by the image of an embedded curve in $\R^d$. 
		Let $\gamma:[a,b] \rightarrow \R^d$ be such a curve and put $M := \gamma([a,b])$. Then $\partial M = \{\gamma(a) , \gamma(b) \}$ and the tangent space at 
		$\gamma(a)$ and $\gamma(b)$ is simply the zero space, that is, 
		$T_{\gamma(a)}(\partial M) = T_{\gamma(b)}(\partial M)  = \{0\}$.
		Since $T_{\gamma(a)}(\partial M)^\bot = T_{\gamma(a)}(\partial M) \oplus  T_{\gamma(a)}(\partial M)^\bot = T_{\gamma(a)}(\partial M)$ we obtain 
		$\nu(a ) = \gamma'(a) / |\gamma'(a)| $ and similarly $\nu(b) =-\gamma'(b)/ |\gamma'(b)|$. Compare also Section~\ref{sec:5} and \cite{MR1803575,Fremiotthesis} for the special case $m=1$ and $d=2$. 
	\end{enumerate}
\end{rems}

\section{The classical structure theorem revisited}\label{sec:3}
The following theorem is commonly known as structure theorem; cf.  \cite[Thm. 3.6, Cor. 1, pp. 479--481]{MR2731611}.  It gives the 
general structure of first order shape derivatives of shape functions 
defined on open or closed  subsets $\Omega$ of $\R^d$. We make use of the 
notation and material introduced in Subsection \ref{sec:quotient_space}.

\begin{thm}[General case]\label{thm:structure_theorem-1}
Let the hold-all $D\subset \R^d$ be open and bounded.  Let $\Omega$ be an open or closed set $\Omega \subset D$ with boundary $\Gamma:=\textnormal{fr}(\Omega)$.
Fix $1 \le k <\infty$.  Suppose that the Eulerian semi-derivative  $dJ(\Omega)(X)$ exists for all 
 $X\in C^k_c(D,\R^d)$.  
\begin{itemize}
\item[(i)] In general we have
	$
	dJ(\Omega)(X) = 0$ for all $ X\in 
	C^k_c(D,\R^d)$	with $ X=0$ on  $\Gamma. $
\item[(ii)] If $X\mapsto dJ(\Omega)(X)$ is linear, then  there is a linear  mapping 
$\tilde  g: \text{im}(\tilde{\mathfrak{J}}_\Gamma) \rightarrow \R$ such that
\ben\label{stru:very_general}
dJ(\Omega)(\VV)=\tilde g(\VV_{|\Gamma})
\een
for all $X\in C^k_c(D,\R^d)$, where $\text{im}(\tilde{\mathfrak{I}}_\Gamma):= \{ \tilde{\mathfrak{I}}_\Gamma(X)|\; X\in Q^k(\Gamma)\}$
denotes the image of $\tilde{\mathfrak{I}}_\Gamma$. 
\item[(iii)] If $ \Omega$ is of class $C^k$ and $dJ(\Omega)$ is of order $k \ge 1$, then  $\text{im}(\mathfrak I_\Gamma) = C^k(\Gamma,\R^d)$ and $\tilde g:C^k(\Gamma,\R^d) \rightarrow \R$ is
a continuous functional. 
\end{itemize}
\end{thm}
\begin{proof}
(i)  Let  $X\in C^k_c(D,\R^d)$ be such that $X=0$ on $\fr(\Omega)$.  
If $\Omega$ is closed, then $\pm X(x) \in T_\Omega(x)$ for all 
$x\in \Gamma$ by definition and obviously  $\pm X(x) \in T_\Omega(x) = \R^d 
$ for all $ x\in \text{int}(\Omega)$. 
So it  follows from Corollary \ref{cor:nagumo} that 
$\Phi_t(\Omega) = \Omega$ for all $t$.   On the other hand if 
$\Omega$ is open, then it follows from Corollary \ref{cor:nagumo-1} 
that $\Phi_t(\Omega) = \Omega$ for all times $t$. 
So in either cases  $dJ(\Omega)(X) = 0$ for all 
$X\in C^k_c(D,\R^d)$ with $X=0$ on $\Gamma$. 
	
(ii) Let $\Omega$ be an open or closed subset of $\R^d$ and fix an integer $k\ge 1$. The set
$T^k(\Gamma)$ is a closed subspace of the vector space 
$C^k_c(D,\R^d)$. Accordingly, the quotient 
$ Q^k(\Gamma) :=  C^k_c( D,\R^d)/T^k(\Gamma)$ is 
well-defined. By item $(i)$ and the linearity of $X\mapsto dJ(\Omega)(X)$ the induced mapping 
$\widetilde{dJ(\Omega)}: Q^k(\Gamma) \rightarrow \R$ is 
well-defined. 
We define the function $\tilde g:\text{im}(\tilde{\mathfrak{J}}_\Gamma) \rightarrow \R$ by the following comuting diagram. 
\begin{diagram}[loose,height=2em,width=80pt]
	C^k_c(D,\R^d) 			& 	 \rTo^{ \mathfrak J_\Gamma}	& 		 \text{im}(\tilde{\mathfrak{J}}_{\Gamma})   &       \\
	\dTo^{\pi}		 	&   	\ruTo^{\tilde{\mathfrak{J}}_\Gamma} \ldTo(2,4)^{\tilde g}			       	    &         &		 \\
	 Q^k(\Gamma)			 & 				    	  & 					    &             \\ 
	 \dTo^{\widetilde{dJ(\Omega)}}	&  					 & 					&\\
		\R    			 &  					&  			    		&
 \end{diagram}
By definition 
$ \tilde g \circ \tilde{\mathfrak{J}}_\Gamma = \widetilde{dJ(\Omega)}$.  
Now $\tilde{\mathfrak{J}}_\Gamma$
is injective and hence invertiable on $\text{im}( \tilde{\mathfrak{J}}_\Gamma )$. Therefore we 
obtain 
$ \tilde g = \widetilde{dJ(\Omega)} \circ \tilde{\mathfrak{J}}_\Gamma^{-1} $
and by definition $dJ(\Omega)(X) = \tilde g(X|_{\Gamma}) $. 
\\
(ii) Now suppose that $\Gamma = \fr(\Omega)$ is of class $C^k$, $k\ge 1$. 
Denote by $E:C^k(\Gamma,\R^d) \rightarrow C^k_c(D,\R^d)$ the extension 
operator.  Then it is readily seen that  
$\tilde{\mathfrak{J}}_\Gamma^{-1} =  \pi \circ E  $, so that 
$\tilde{\mathfrak{J}}_\Gamma:Q^k(\Gamma) \rightarrow C^k(\Gamma, \R^d) $ is 
 surjective.  Hence we get $\text{im}(\tilde{\mathfrak{J}}_\Gamma) = C^k(\Gamma, \R^d)$.  
 From this 
it follows that $\tilde g$ is a linear and continuous functional 
on $C^k(\Gamma,\R^d)$. 
\end{proof}

Nagumo's theorem allows us to show that the distribution given by 
\eqref{stru:very_general}  depends explicitly  on normal perturbations 
$X\cdot \nu$ if we require the boundary to be smoother.  

\begin{cor}[Smooth case]\label{thm:structure_theorem} 
 Let $\Omega$ be open in $\R^d$ with a compact  $C^{k+1}$-boundary 
 $\Gamma:= \fr(\Omega)$,  $1 \le k < \infty $. Suppose that $J$ is shape differentiable at 
 $\Omega$ and that 
$dJ(\Omega)$ is of order $k$. Then there exists a linear and continuous function $g: C^k(\Gamma)\rightarrow \R$, such that
\ben\label{volume}
dJ(\Omega)(X)=g(X_{|\Gamma}\cdot \nu) \quad \text{ for all } X\in C^k_c(D,\R^d).
\een
\end{cor}
\begin{proof}
As $\Gamma:=\fr(\Omega)$ is of class $C^k$, $k\ge 1$, we know by 
Theorem~\ref{thm:structure_theorem-1} that 
there is a linear and continuous functional 
$\tilde  g:C^k(\Gamma, \R^d) \rightarrow \R$, such that  
\ben\label{eq:derivative_boundary}
 dJ(\Omega)(X) = \tilde g(X|_{\Gamma}) \quad \text{ for all } X\in C^k_c(D,\R^d).
 \een
We split $X$ into normal and tangential part along $\Gamma$, that is, $X_{|\Gamma}=X_{\mathfrak t} + (X_{|\Gamma}\cdot \nu)\nu$, where $\nu$ is the 
normal vector along $\Gamma$ and $X_{\mathfrak t}:= X_{|\Gamma} -  (X_{|\Gamma}\cdot \nu)\nu$. As the 
boundary $\Gamma$ is of class $C^{k+1}$ the normal $\nu$ is of class 
$C^k(\Gamma,\R^d)$. Then 
it follows from Corollary \ref{cor:nagumo-1}
 that $dJ(\Omega)(X)=0$ for all 
$X$ in $C^k_c(D,\R^d)$ with $X\cdot \nu =0$ on $\Gamma$.   
Therefore 
extending $\nu$ to a function $\tilde \nu \in C_c^k(D, \R^d)$ and defining 
$\tilde X_{\mathfrak t} := X+ (\tilde \nu\cdot X) \tilde \nu$ shows that 
$0 = dJ(\Omega)(\tilde X_{\mathfrak t}) = g(X_{\mathfrak t}). $
So inserting $X$ into \eqref{stru:very_general}, we find 
\ben
dJ(\Omega)(X) = \tilde g(X_{\mathfrak t}) + \tilde g((X_{|\Gamma}\cdot \nu)\nu) = \tilde g((X_{|\Gamma}\cdot \nu)\nu).
\een
The mapping $g(v):=  \tilde g(v\nu)$ is continuous on  $C^k(\Gamma)$. 
\end{proof}

\begin{rems}
	\begin{enumerate}
\item[(a)] Corollary~\ref{thm:structure_theorem} is usually referred to as  structure theorem.
\item[(b)]
If  $g$ in  Theorem \ref{thm:structure_theorem} belongs to $L_1(\Gamma)$, then we have the typical boundary expression
\ben\label{eq:bound_exp}
dJ(\Omega)(\VV)=\int_\Gamma g\, \VV\cdot \nu\, ds.
\een
This expression of the derivative  is usually referred to as  {\it Hadamard} or {\it Hadamard-Zol\'esio} formula. 
\item[(c)] It is important to note that 
	if one wants a formula like \eqref{volume} for the shape derivative, 
	then  the smoothness of  the boundary $\fr(\Omega)$ has to be one 
	order higher than the order $k$ of $dJ(\Omega)$. 
The reason is that in 
order to have the unit normal vector field  in $C^k$, we need the boundary $\fr(\Omega)$ to be of 
class $C^{k+1}$. However, to obtain that the derivative actually ``lives'' on 
the boundary it  is no regularity on the boundary nessacary.  
In less regular situations, that is, when $\Omega$ has less regularity,  it is still possible to obtain a formula 
in the spirit of \eqref{volume}. However, this requires notions from 
 geometric measure 
theory; cf.  \cite{LamPierr06}.
\end{enumerate}
\end{rems}


\section{Structure theorem for $C^k$-submanifold} \label{sec:4}
 In this section, we study the 
structure of the shape derivative of  real-valued shape 
functions $$ J:\Xi \rightarrow \R, \qquad M \mapsto J(M),  $$
where $\Xi\subset \mathcal A^k_m$, $1\le k,m < \infty$, is some admissible set and 
$$\mathcal A^k_m = \{M \subset \R^d| M \text{ is closed and bounded } m\text{-dimensional }  C^k\text{-submanifold of } \R^d\}. $$

\subsection{Splitting of vector fields}
Let $M$ be a $m$-dimensional  closed and bounded $C^k$-submanifold of $\R^d$. We use the notation 
$ \mathfrak X^k(M) :=\{X:M\rightarrow TM|\, X \text{ is of class } C^k \}$
for the space of $C^k$-vector fields on $M$. Similarly, 
$\mathfrak X^k_\bot (M)$ denotes the normal fields along $M$.
 We introduce the orthogonal projection 
$\mathfrak p_{T_pM}:T_p\R^d \rightarrow T_pM$ by 
$$ (\mathfrak p_{T_pM} (X)-X,V) =0 \quad \text{ for all } V\in T_pM. $$   
Then defining $\mathfrak p_{T_pM}^\bot(x) := x - \mathfrak p_{T_pM}(x)$ we have
$\ker{\mathfrak p_{T_pM}} = \im{\mathfrak p_{T_pM}^\bot} = (T_pM)^\bot$. 
Note that the projection depends on $p\in M$ as the tangent space 
varies when $p$ changes.

Now given a function  $X \in C^k(M,\R^d)$, we define its 
{\bf orthogonal projection} onto the vector  bundle $TM^\bot := \cup_{p\in M}(T_pM)^\bot$ pointwise by 
$$ X \mapsto \mathfrak p_{T_pM}^\bot(X)|_p =: X^\bot_p.   $$
This defines a mapping 
$$  \mathfrak p_{T_pM}^\bot: C^k(M,\R^d) \rightarrow \mathfrak X^k_\bot (M).$$
As $\R^d = T_p\R^d = T_pM \oplus (T_pM)^\bot $ for all $p$ in $M$, we may 
write $X = \tilde X + \tilde X^\bot$, where $\tilde X\in \mathfrak X^k(M)$ and 
$\tilde X^\bot\in \mathfrak X^k_\bot(M)$ and  then by 
definition 
$\mathfrak p_{T_pM}^\bot(X) = \tilde X^\bot $.  
\begin{defin}
	A subset $S\subset M$ of a $m$-dimensional submanifold is called embedded submanifold if 
	the inclusion map
	$ i: S\rightarrow M, \; x \mapsto x  $
	is an embedding. We call $S$ closed, embedded submanifold if $i:S\rightarrow M$ is proper, that is, $i^{-1}(A)$ is compact for all $A\subset M$ compact. 
\end{defin}
The following two lemmas will be crucial for our investigation.
The first one can be established using local charts; cf.  \cite{Lee}.
\begin{lem}\label{lem:extension_vector_field}
Let $M$ be a $m$-dimensional $C^k$-submanifold $M$ and let 
$S\subset M$ be a $s$-dimensional closed, embedded submanifold of $M$.  
Then every vector field $X\in \mathfrak X^k(S)$ can be extended to $M$, that is, there is a vector field $\tilde X\in \mathfrak X^k(M)$ satisfying
 $\tilde X|_S = X$.
\end{lem}
\begin{rem}
\begin{itemize}
\item If $M$ is a $m$-dimensional submanifold with boundary $\partial M$, then $\partial M$ is 
a closed, embedded submanifold of $M$. Hence every vector field defined on the boundary can be extended to all of $M$. 
\item If $M=\R^d$ and $S\subset \R^d$ is a closed, embedded $C^k$-submanifold, then every vector field defined on 
$S$ can be extended to a $C^k$-vector field on $\R^d$. In particular, if $S$ is compact then 
the support of the vector field can be chosen to lie in some open set $D$ 
containing $S$. 
\end{itemize}
\end{rem}

\begin{lem}\label{lem:splitting_vector_field}
Let $M$ be a closed and bounded $m$-dimensional  $C^{k+1}$-submanifold of 
$\R^d$ contained in an open set $D\subset \R^d$.  Let us denote 
by $\nu$ the unique outward-pointing  unit vector 
field on $\partial M$. 
Then to each vector field  $X\in C^k_c(D,\R^d)$, 
we find vector fields $X^\bot, X^{\mathfrak t}, X^{\nu}\in C^k_c(D,\R^d)$ satisfying
$$ X = X^\bot + X^{\mathfrak t} + X^{\nu} \quad \text{ in } M, $$
 and
\begin{align}
X^{\mathfrak t}_p  \in T_pM \quad &\text{ for all } p\in M, \label{eq:tangent-1}\\
X^{\mathfrak t}_p \in T_p(\partial M) \quad &  \text{ for all } p\in \partial M,\\
X^{\nu}_p = (X\cdot \nu)\nu  \quad &  \text{ for all } p\in\partial M,\\
X^\bot_p  \in (T_pM)^\bot \quad & \text{ for all } p\in M.\label{eq:tangent-4}
\end{align}  
\end{lem}
\begin{proof}
At first, we define 
\ben\label{eq:X_t}
\hat  X^{\mathfrak e}_p := X_p - \hat  X^\bot_p 
 \een
for all $p\in M$, where $\hat X^\bot = \mathfrak p_{T_pM}^\bot (X|_{M})$.  
Since by definition  we have the decomposition 
$ \R^d = (T_pM)^\bot \oplus T_pM  $ for all points $p\in M$, we obtain $\hat  X^{\mathfrak e}_p\in T_pM$ for all $p$ in $M$. 
By definition we have 
$$ (\hat X^{\mathfrak e}_p - X_p, v_p) = 0 \quad \text{ for all }v_p\in T_pM, \quad  p\in M,$$
and this shows that $\hat X^{\mathfrak e} \in C^k(M, \R^d)$ by using local 
charts.  It follows that $\hat X^{\bot}\in C^k(M,\R^d)$.
According to Lemma~\ref{lem:extension_vector_field}, we may extend $\hat X^{\mathfrak e}$ and $\hat X^\bot$ to functions 
$ X^{\mathfrak e}, X^\bot \in C^k_c(D,\R^d)$. 
On the other hand we have for all boundary points $p\in \partial M$
\ben\label{eq:decomposition}
 \R^d = (T_pM)^\bot \oplus T_pM = (T_pM)^\bot \oplus T_p(\partial M) \oplus^{T_pM} (T_p(\partial M))^\bot. 
 \een
Denote by $\nu\in \mathfrak X^k(\partial M)$ the outward-pointing unit normal 
field along $\partial M$. 
 As the injection $i:\partial M  \rightarrow M$ is proper we 
may apply Lemma~\ref{lem:extension_vector_field} and extend the vector field 
$\hat X^\nu := (X \cdot \nu)\nu $ on $\partial M$, 
to a vector field $X^\nu \in\mathfrak{X}^k( M)$, which itself can be extended 
to a vector field in $C^k_c(D, \R^d)$, (still keeping the same notation).
Finlly, we  put 
$$ X^{\mathfrak t} :=  X^{\mathfrak e} -  X^{\nu} \stackrel{\eqref{eq:X_t}}{=} X-  X^\bot -  X^{\nu}. $$
In view of the fact that $T_pM$ is a linear space, we obtain 
$ X^{\mathfrak t}|_M = (\hat X^{\mathfrak e} - \hat X^{\nu})|_{M}\in T_pM$ for 
all $p\in M$ and because of \eqref{eq:decomposition} it follows 
$X^{\mathfrak t}_p \in T_p(\partial M)$ for all $p \in M$.  Hence we obtain 
the required decomposition
$ X = X^{\mathfrak t} +  X^\bot + X^{\nu}$ on $M$. 
\end{proof}

\subsection{The structure theorem for submanifolds}
With the preparations of the previous section we are now able to state our 
main result. 
\begin{thm}[Structure theorem for submanifolds]\label{thm:structure} 
Let $M$ be a bounded and closed $m$-dimensional $C^{k+1}$-submanifold of $\R^d$ 
contained in a bounded and open set $D\subset \R^d$. Suppose that $J$ is 
shape differentiable at $M$ and assume that $dJ(M)$ is 
of order $k$. Then there exist continuous functionals 
$h: C^k(M, \R^d) \rightarrow \R$ and  $g: C^k(\partial M) \rightarrow \R$
such that 
\ben\label{eq:struc_manifold}
dJ(M)(X) = h(X^\bot) + g(X_{|\partial M}\cdot \nu) \quad \text{ for all } X\in C^k_c(D,\R^d), 
 \een
where $X^\bot := \mathfrak p_{T_pM}^\bot(X|_{M})$ and 
$\nu$ is the unique outward-pointing unit vector field along $\partial M$. 
\end{thm}
\begin{proof}
Recall definition \eqref{def:vector_zero}
namely $ T^k(M) = \{X\in C^k_c(D,\R^d)|\, X|_{M} = 0 \} .$
The set  $T^k(M)$ is a closed linear subspace of $C^k_c(D,\R^d)$.  
Recall the definition of the quotient 
$\mathcal Q^k(M) =  C^k_c(D,\R^d)/T^k(M) $
and consider similarly to Theorem~\ref{thm:structure_theorem-1} we have a 
commuting 
diagram: 
\begin{diagram}[loose,height=2em,width=80pt]
	C^k_c(D,\R^d) 			& 	 \rTo^{ \mathfrak J_M}	& 		 C^k(M,\R^d)   &       \\
	\dTo^{\pi}		 	&   	\ruTo^{\tilde{\mathfrak{J}}_M} \ldTo(2,4)^{\tilde h}			       	    &         &		 \\
	 Q^k(M)			 & 				    	  & 					    &             \\ 
	 \dTo^{\widetilde{dJ(M)}}	&  					 & 					&\\
		\R    			 &  					&  			    		&
 \end{diagram}
 By definition we have 
 $\widetilde{dJ(M)} = \tilde h \circ \mathfrak{\tilde{J}}_M $. 
 We see that  $ \mathfrak{\tilde{J}}_M^{-1} = \pi \circ E$, where
 $E:C^k(M,\R^d) \rightarrow C^k_c(D,\R^d)$ denotes the  continuous extension 
 operator.  It follows that 
 $\tilde h = \widetilde{dJ(M)} \circ  \mathfrak{\tilde{J}}_M^{-1}:C^k(M,\R^d) \rightarrow \R$
 is continuous. 
 By construction 
 $
\tilde h(X|_{M}) = \tilde h( \mathfrak{\tilde{J}}_M(\pi(X)) ) =   
\widetilde{dJ(M)} \circ\pi(X)  =  dJ(M)(X).
$
 Now we apply Lemma~\ref{lem:splitting_vector_field} to $X$ and find $X^\bot, X^\nu, X^{\mathfrak t}\in C^k_c(D,\R^d)$ satisfying \eqref{eq:tangent-1}-\eqref{eq:tangent-4} and 
$ X= X^\bot +  X^\nu + X^{\mathfrak t} $ on $M$. 
As $X^{\mathfrak t}_p\in T_pM$ for all $p\in M$ and $X^{\mathfrak t}_p\in T_p(\partial M)$ for all $p\in \partial M$,
we get from Proposition~\ref{prop:nagumo} that $0 = h(X^{\mathfrak t}) =  dJ(M)(X^{\mathfrak t})$, which implies
\ben\label{eq:first_step}
 dJ(M)(X) =   h(X^{\nu}|_{M}) + h(X^\bot) \quad \text{ for all }X\in C^k_c(D,\R^d).
 \een
Consequently 
$ dJ(M)(X) =  h(X^{\nu}|_{M})$ for all  $X\in C^k_c(D, \R^d)$  with   $X|_{M} \in \mathfrak X^k(M). $
To further process the right hand side of \eqref{eq:first_step} we introduce the  linear and continuous mapping 
$$
\mathcal{I}_{\nu}:\mathfrak X^k(M)\rightarrow  C^k(\partial M), \quad [X] \mapsto   X|_{\partial M}\cdot \nu  $$
and the linear space 
$ \tilde T^k(M) := \{X\in \mathfrak X^k(M): X|_{\partial M}\cdot \nu = 0 \text{ on } \partial M \}.   $
This space is a linear subspace of $\mathfrak X^k(M)$.  We define $g:\mathfrak X^k(M) \rightarrow \R$ by letting the following diagram 
commute.  

\begin{diagram}[loose,height=2em,width=80pt]
	\mathfrak X^k(M)		& 	 \rTo^{\mathcal I_{\nu}}  	& 		 C^k(\partial M)   &       \\
	\dTo^{\pi}		 	&   	\ruTo^{ \tilde{\mathcal{I}}_{\nu}   } \ldTo(2,4)^{ g}			       	    &         &		 \\
	 \mathfrak X^k(M)/\tilde T^k(M)   & 				    	  & 					    &             \\ 
	 \dTo^{\tilde h}	&  					 & 					&\\
		\R    			 &  					&  			    		&
 \end{diagram}
 By definition $\tilde h = g\circ \tilde{\mathcal{I}}_\nu$
on $ \mathfrak X^k(M)/\tilde T^k(M)$. 
As before we extend $\nu:\partial M 
\rightarrow TM$  to a vector field 
$\tilde \nu :M \rightarrow TM$. Moreover denote by 
$\tilde E:C^k(\partial M) \rightarrow C^k(M)$ the usual extension
operator.  Then we  see that 
$ \tilde{\mathcal{I}}_\nu^{-1}(f) = \pi \circ  ( \tilde \nu \cdot \tilde E(f))   $
. 
Therefore $g:C^k(\partial M) \rightarrow \R$ is continuous and
for every $X\in \mathfrak X^k(M)$, we obtain
\begin{align*}
 g( X|_{\partial M}\cdot \nu) & =  g( \tilde{\mathcal{I}}_\nu (\pi (X|_{M} ) )) \\
 & =   \tilde h( \pi(X|_{M}  )  )  \\
 & = h(X|_{M} )
\end{align*}
and thus $g( X|_{\partial M}\cdot \nu) = h(X^\nu|_{M} ) = dJ(M)(X^\nu|_{M}) $. 
Plugging this into \eqref{eq:first_step} we recover 
\eqref{eq:struc_manifold}. The continuity of $g$ follows from the continuity 
of the extension operator. 
\end{proof}

We conclude this section with the following two special cases of our main result.
\begin{cor}\label{cor:structure1} 
Let $M$ be a closed and bounded  $m$-dimensional $C^{k+1}$-submanifold of $\R^d$ 
without boundary, that is,  $\partial M = \emptyset$. Suppose that $J$ is 
shape differentiable at $M$ and assume that $dJ(M)$ is 
of order $k$. Then there exists a continuous functional 
$h: C^k(M,\R^d) \rightarrow \R$, such that 
$$ dJ(M)(X) = h(X^\bot) \quad \text{ for all } X\in C^k_c(D,\R^d). $$
\end{cor}

\begin{cor}\label{thm:structure1} 
Let $M$ be a closed and bounded $d$-dimensional $C^{k+1}$-submanifold of $\R^d$. Suppose that $J$ 
is 
shape differentiable at $M$ and assume that $dJ(M)$ is 
of order $k$. Then there exists a continuous functional 
$ g: C^k(\partial M) \rightarrow \R$, such that 
$$ dJ(M)(X) =  g(X_{|\partial M}\cdot \nu)\quad \text{ for all } X\in C^k_c(D,\R^d). $$
\end{cor}

\section{Application to shape functions}\label{sec:5}
\subsection{Shape functions defined on smoothly cracked sets}
Cracked sets naturally  arise in fracture mechanics, where they model damage 
of  solids; cf. \cite{Fremond}. Cracked sets are highly irregular and do not even satisfy the 
cone property, but the crack itself is often assumed to be Lipschitz 
continuous or smoother. In order to forcast the propagation of a crack it is essential 
to compute shape derivative in cracked sets. For PDE constrained shape 
functions, the derivation of 
the shape differentiability at a cracked set 
\cite{crack_hoem,crack_soko,crack_soko2} or smooth sets \cite{delsturm,sturm0,sturm3,desaintzolesio,sokozol,ferchichizolesio} is a challenge 
itself. Here we are interested in the exact 
structure of the shape derivative in cracked sets and will assume that the 
shape function is shape differentiable.

\begin{defin}\label{def:crack}
	Let $\Omega\subset \R^d$ be an open and bounded set. 
	\begin{itemize}
		\item[(i)] The set $\Omega$ is called {\bf crack free} if  
	$\text{int}(\overline \Omega) = \Omega$, otherwise we call $\Omega$ cracked. 
\item[(ii)] The set $\Omega\subset \R^d$ is said to be {\bf smoothly $l$-cracked}, 
	$l\ge 1$, if there is 
an open subset $\tilde \Omega\subset \R^d$ with $C^k$-boundary 
$\fr(\tilde \Omega)$, $k\ge 1$, and a  closed, bounded and simply connected $l$-dimensional $C^k$-submanifold $\Sigma \subset \tilde \Omega$ of $\R^d$, such that  
$\Omega =\tilde \Omega \setminus \Sigma $.
\end{itemize}
\end{defin}
\begin{rem}
Note that every  open subset $\Omega \subset \R^d$ with $C^k$-boundary 
$\fr(\Omega)$ is crack-free, so that part (ii) of Definition 
\ref{def:crack} makes sense. In particular, a smoothly cracked set 
can not have any further cracks except $\Sigma$. 
\end{rem}
Now we want to verify that the shape derivative in a smoothly cracked set 
can be obtained as the shape derivative of a shape function depending on the  
only depending the  on the crack itself.

\begin{lem}\label{lem:crack_vs_manifold}
Suppose that $\Omega \subset \R^d$ is smoothly $l$-cracked of class $C^k$ with $C^k$-set 
$\tilde \Omega\subset \R^d$ and a $l$-dimensional $C^k$-submanifold 
$\Sigma\subset \R^d$, $k\ge 1$, such that  $\Omega = \tilde  \Omega \setminus \Sigma$. 
Set $M:= \Sigma$.  Let $\Omega \mapsto J(\Omega)$ be a shape functions and define $\tilde J(M):= J(\Omega \setminus M)  $.  
Then
$$d\tilde J(M)(X) = 	 dJ(\Omega)(X),$$
where $X\in C^k_c(\tilde \Omega, \R^d),$ if either of the two expressions exists. 
\end{lem}
\begin{proof}
As $X\in C^k_c(\tilde \Omega, \R^d)$ it evident that for all $t$
$$ \Phi_t(\Omega) = \Phi_t(\tilde \Omega\setminus \Sigma) = \Phi_t(\tilde \Omega) \setminus \Phi_t(\Sigma) = \tilde \Omega \setminus \Phi_t(\Sigma)	.	$$
From this the conclusion of the lemma follows.  
\end{proof}
This lemma shows that shape functions depending on  smoothly cracked sets can be 
seen as shape functions only depending on the crack itself. Next, we consider the special situation of a shape function defined on smoothly $1$-cracked sets in $\R^2$; cf. \cite{MR1803575}.  
\begin{lem}
	Let $\Omega$ be a smoothly $l$-cracked subset of $\R^2$ of class $C^2$. 
	By definition there are an open and bounded set $C^2$-set $\tilde \Omega \subset \R^2$ and a 
	 closed, bounded, and simply connected $l$-dimensional submanifold 
	$\Sigma\subset \tilde \Omega$ of class $C^2$, such that 
	$\Omega = \tilde \Omega \setminus \Sigma$.  Set $M:=\Sigma$, $\partial M = \{A,B\}$, and 
	suppose that  $dJ( \Omega):C^1_c(\tilde \Omega, \R^2) \rightarrow \R$ 
	is linear and  continuous. Then there are real numbers 
$\alpha_1, \alpha_2$ and a linear and continuous functional $\bar h:C^1(M)\rightarrow \R$, such that  
\ben\label{eq:alpha_crack}
dJ(\Omega )(X)= \alpha_1 (X\cdot \nu)(A) + \alpha_2 (X\cdot \nu)(B) + \bar 
h(X_{| M} \cdot \mathfrak n)
\een
for all $X\in C^k_c(\tilde \Omega,\R^2)$, where $\mathfrak n$ is a unit normal field along $M$ and $\nu$ the unit normal vector field on 
$\partial M$. 
\end{lem}
\begin{proof}
Taking into account 
Lemma~\ref{lem:crack_vs_manifold} we see that we can
apply Theorem \ref{thm:structure} to 
$M\mapsto \tilde J(M) := J(\tilde \Omega \setminus M)$ and obtain linear functionals 
$g:C^1(\partial M )\rightarrow \R$ and $h:C^1(M,\R^d)\rightarrow \R$, such that 
\ben\label{eq:struc_crack}
d\tilde J(M)(X) = g(X_{|\partial M}\cdot \nu) + h(X^\bot).
\een
We have $\partial M = \{A,B\}$ and thus 
$C^k(\partial M) = \{ f: \{A,B\}\rightarrow \R \}$. 
We may define a basis $f_1,f_2:\partial M \rightarrow \R$ of $C^k(\partial M)$ 
by $f_1(A) := 1$, $f_1(B) := 0$
and $f_2(A) := 0$, $f_2(B) := 1$. Then every 
$f\in C^k(\partial M)$ can be written as $f= \alpha_1 f_1 + \alpha_2 f_2$.  
In particular, we have 
$X_{|\partial M}  \cdot \nu =   (X\cdot \nu)(A) f_1 +   (X\cdot \nu)(B) f_2$
 so that 
 \ben\label{eq:g_crack}
 g(X_{|\partial M} \cdot \nu) =\alpha_1 (X\cdot \nu)(A) +  \alpha_2 (X\cdot \nu)(B),
 \een
where $\alpha_1 :=g(f_1) $ and $\alpha_2 := g(f_2)$. 
 Denote by 
$\mathfrak n$ the unit normal field along $M$.  Then 
$X^\bot_{| M} = (X_{|M}  \cdot \mathfrak n) \mathfrak n $ 
and thus 
\ben\label{eq:h_crack}
h(X^\bot_{| M}) = h((X_{| M}\cdot \mathfrak n) \mathfrak n)
\een
for all $X\in C^k_c(\tilde \Omega, \R^2)$. 
Setting $\bar h(v) := h(v\mathfrak n)  $, we recover \eqref{eq:alpha_crack}.
\end{proof}

\begin{rem}
We may describe the crack $\Sigma$ by an embedded curve 
$\gamma:[a,b] \rightarrow \R^d$ of class $C^2$, that is,
$\gamma([a,b])  =: \Sigma\subset \tilde \Omega$ and $\gamma(a)=A$ and 
$\gamma(b) = B$. 
Then $\nu\circ \gamma(a) = \gamma'(a)/|\gamma'(a)|$ and 
$\nu\circ \gamma(b) = -\gamma'(b)/|\gamma'(b)|$. 
\end{rem}

\begin{cor}
Let $ \Omega \subset D\subset \R^d$ be a smoothly $1$-cracked set  such that  
$\Omega = \tilde \Omega \setminus \Sigma$, where $\tilde \Omega$ is an open 
and bounded  set of 
class $C^\infty$ and $M:=\Sigma$ is a closed, bounded and simply connected 
$l$-dimensional submanifold of $\R^d$ of class $C^\infty$. Let $J$ be a 
shape function and suppose that  
$dJ(\Omega):C^1_c(\tilde \Omega, \R^d) \rightarrow \R$ is continuous and 
linear. 
Let $X \in C^1_c(D,\R^d)$. Then there are continuous and linear functionals 
$\bar h_1, \ldots, \bar h_{d-1}:C^1(M)\rightarrow \R$ and real numbers 
$\alpha_1, \alpha_2$ such that
\ben\label{eq:crack_dim_one}
dJ(\Omega )(X) = \alpha_1 (X\cdot \nu)(A) + \alpha_2 (X \cdot \nu)(B) + 
\sum_{i=1}^{d-1}\bar h_i(X_{| M}\cdot \mathfrak n_i)
\een
for all $X\in C^k_c(\tilde \Omega, \R^d)$, where $(\mathfrak n_1, \ldots, \mathfrak n_{d-1})$ is an orthonormal frame along $M$ 
satisfying\\
$\text{span}\{\mathfrak n_1(p), \ldots, \mathfrak n_{d-1}(p)\} = (T_pM)^\bot$ for all $p\in M$. 
\end{cor}
\begin{proof}
From the previous lemma, we obtain
$ dJ(\Omega)(X) = \alpha_1 (X\cdot \nu)(A) + \alpha_2 (X\cdot \nu)(B) + \bar h(X^\bot),
$
and hence taking into account 
$X^\bot = ( X  \cdot \mathfrak n_1 )\mathfrak n_1 + \cdots + ( X \cdot \mathfrak n_{d-1} )\mathfrak n_{d-1}  $
we arrive at 
$ dJ(\Omega)(X) = \alpha_1 (X\cdot \nu)(A) + \alpha_2 (X\cdot \nu)(B) + \sum_{i=1}^{d-1} \bar h( (X \cdot \mathfrak n_i ) \mathfrak n_i)$.
So setting $\bar h_i(v) := h(v\mathfrak n_i) $, we recover \eqref{eq:crack_dim_one}.

\end{proof}

\subsection{Shape functions defined on submanifolds of dimension one and two}

\paragraph{Length variation of a curve in $\R^3$}
Let $\gamma:[a,b] \rightarrow \R^3$ be an embedded curve of class $C^2$ so that 
$M:= \gamma([a,b])$ becomes a one dimensional $C^2$-submanifold of $\R^3$ with boundary 
$\partial M =\{\gamma(a), \gamma(b)\}$.
 We consider the shape function 
 $$ J(M) := \int_a^b |\gamma'(t)|\,dt   $$
and  denote by 
 $
 T(t):= \gamma'(t)/  | \gamma'(t)|,$ $ N(t) := T'(t)/ |T'(t)|$, and $ B(t) 
 := T(t)\times N(t)  $
the tangential, normal and binormal  vector field along $\gamma$, respectively. If $\gamma$ 
is arc-length parametrised, then we define the curvature $\kappa$ of $\gamma$  by 
$T' = \kappa N$.  If $\gamma$ is not arc-length parametrised, then we 
have 
 $ T'   =  v\kappa N$
 on $[a,b]$, where $v(t) := |\gamma'(t)|$. Let 
$\mathfrak n,\mathfrak t, \mathfrak b:M\rightarrow \R^3$ be  unit vector fields, such that 
$ T = \mathfrak t \circ \gamma$,  $N= \mathfrak n\circ \gamma$, and $ B =  \mathfrak b \circ \gamma $.
\begin{lem}
Let $D$ be an open and bounded set of $\R^3$ containing $M$ and let 
$X \in C^2_c(D,\R^3)$. 
Then
$$dJ(M)(X) =  \int_a^b \frac{ \gamma'(t) \cdot (\partial X \circ \gamma(t)) \gamma'(t)}{|\gamma'(t)|} \,dt $$
or equivalently
\ben\label{eq:structure}
 dJ(M)(X) = \int_a^b  \kappa(t)  (X^\bot \cdot \mathfrak n)\circ \gamma(t) \; |\gamma'(t)|  
 \,dt + (X \cdot \mathfrak t)(\gamma(b))-(X\cdot \mathfrak t)(\gamma(a)), 
 \een
 where 
 $X^\bot  = (X\cdot \mathfrak n) \mathfrak n +  (X\cdot \mathfrak b)  \mathfrak b$. 
 
\end{lem}
\begin{proof}
 We compute
\ben\label{eq:per}
\begin{split}
dJ(M)(X)  =& \frac{d}{ds} \left( \int_a^b |(\partial \Phi_s\circ \gamma(t)) 
\gamma'(t)|\, dt \right)\bigg|_{s=0} \\
=&\int_a^b| \frac{ \gamma'(t) \cdot (\partial X \circ \gamma(t)) \gamma'(t)}{|\gamma'(t)|} \,dt\\
=& \int_a^b (X(\gamma(t))' \frac{ \gamma'(t)  }{|\gamma'(t)|}\,dt\\
=&   \int_a^b |\gamma'(t)| \kappa(t)  (X\circ \gamma(t))\cdot N(t)  \,dt + X(\gamma(b)) \cdot T (b)-X(\gamma(a)) \cdot T (a).
\end{split}
\een
 From this the result follows.  
\end{proof}

\begin{cor}
Let $D$ be an open and bounded set in $\R^3$ containing $M$ and let 
$X \in C^2_c(D,\R^3)$.  Suppose that $\gamma:[a,b]\rightarrow \R^2$ is a simply closed
$C^2$-curve. Then 
$$
dJ(M)(X) = \int_a^b  \kappa  (X^\bot \cdot \mathfrak n)\circ \gamma(t) \; |\gamma'(t)|  
\,dt. $$
 \end{cor}

\paragraph{Variation of the surface integral in $\R^3$}
As a two dimensional example, we consider the variation of the surface 
integral of a  cylinder-like surface in $\R^3$.
We define $Q:=[a,b]\times [c,d]$ and let $\varphi:Q\rightarrow \R^3$ be a $C^2$-embedding and put 
$M:=\varphi(Q)$. We assume that $\varphi(u,c) = \varphi(u,d)$, 
$ \partial_v\varphi(u,c) = \partial_v\varphi(u,d)$ and  
$ \partial_v^2\varphi(u,c) = \partial_v^2\varphi(u,d)$ for all $u\in [a,b]$. 
Since $\varphi$ is an embedding $\varphi_u:=\partial_u \varphi$ and $\varphi_v:=\partial_v \varphi$ are linearly independent at each point 
$(u,v)$ of $Q$. Hence the unit normal 
vector to the surface $M$ is given by 
$$ N(u,v) := \frac{\varphi_u\times \varphi_v}{|\varphi_u\times \varphi_v|}. $$
Recall that the classical surface integral of 
$\varphi(Q)$ is defined by 
$$ J(M)  := \int_a^b\int_c^d |\varphi_u\times \varphi_v|\, dudv. $$ 
\begin{lem}
Let $D$ be an open and bounded set containing $M$. Suppose that $X\in C^2_c(D,\R^3)$. Then 
 \begin{align*}
dJ(M)(X) =  \int_a^b\int_c^d \partial_u(X\circ \varphi)\times \varphi_v \cdot  N \, dudv 
  + \int_a^b\int_c^d   \varphi_u\times \partial_v(X\circ \varphi) \cdot  N \, dudv
\end{align*}
which is equivalent to 
\ben\label{eq:struc_surf}
\begin{split}
dJ(M)(X)  = & \int_a^b\int_c^d H(u,v) X(\varphi(u,v))\cdot  N(u,v)  |\varphi_u\times \varphi_v|\, dudv \\
& +  \bigg[\int_c^d   ( X\cdot \nu)\circ \varphi |\varphi_v|  \, dv\bigg]^a_b,
\end{split}
\een
where $H(u,v)$ is the 
mean curvature at the surface point $\varphi(u,v)$ and 
$\nu$
the outward-pointing unit normal along $\partial M$. 
\end{lem}
\begin{proof}
We compute
\ben\label{eq:per_surf}
\begin{split}
	dJ(M)(X) =&  \dt\left(\int_a^b\int_c^d |(\partial \Phi_t\circ \varphi)\varphi_u\times (\partial \Phi_t\circ \varphi)\varphi_v|\, dudv\right)|_{t=0}  \\
  = &\int_a^b\int_c^d  \frac{(\partial X \circ \varphi)  \varphi_u\times \varphi_v \cdot \varphi_u\times \varphi_v}{|\varphi_u\times \varphi_v|} \, dudv + \int_a^b\int_c^d  \frac{  \varphi_u\times (\partial X \circ \varphi)\varphi_v \cdot \varphi_u\times \varphi_v}{|\varphi_u\times \varphi_v|} \, dudv \\
   =& \int_a^b\int_c^d \partial_u(X\circ \varphi)\times \varphi_v \cdot  N \, dudv 
  + \int_a^b\int_c^d   \varphi_u\times \partial_v(X\circ \varphi) \cdot  N \, dudv \\
  =& -\int_a^b\int_c^d (X\circ \varphi) \times \varphi_v \cdot  N_u \, dudv 
    - \int_a^b\int_c^d   \varphi_u\times (X\circ \varphi) \cdot  N_v \, dudv \\
    & +  \bigg[\int_c^d   (X\circ \varphi)\times \varphi_v \cdot  N \, dv\bigg]^a_b   ,
\end{split}
\een
where we used  $N(u,c)=N(u,d)$ for all $a\le u \le b$ which follows from 
$\partial_v \varphi(u,c) = \partial_v \varphi(u,d)$ for all $a \le u \le b$ . 
Now since $  N^2=1$ we have $ N_u\cdot  N=0$ and $ N_v\cdot  N=0$, 
which means that $ N_u, N_v \in d_p\varphi(T_pQ)$. 
Thus we can write (Weingarten equations) 
\begin{align*}
   N_u &= \alpha_1  \varphi_u +  \alpha_2 \varphi_v \\
     N_v &= \alpha_3 \varphi_u +  \alpha_4 \varphi_v
\end{align*}
for smooth functions $\alpha_i$.  Note that $H(u,v) = \alpha_1 + \alpha_4$ 
(that is the trace of the Weingarten mapping).
Therefore using 
$ (a\times b) \cdot c = (c\times a) \cdot b = (b\times c) \cdot a, $
we get 
\ben\label{eq:X_times_phi}
\begin{split}
 (X\circ \varphi) \times \varphi_v \cdot  N_u  & =-\alpha_1 \varphi_u \times \varphi_v \cdot X \circ \varphi = -\alpha_1 N \cdot  (X \circ \varphi) |\varphi_u \times \varphi_v| \\ 
  \varphi_u\times (X\circ \varphi) \cdot  N_v  & =-\alpha_4 \varphi_u \times \varphi_v \cdot X \circ \varphi = -\alpha_4 N \cdot  (X \circ \varphi) |\varphi_u \times \varphi_v|
  \end{split}
.\een
Note also that the outward-pointing unit normal field $\nu$ satisfies
$\nu\circ \varphi  =  \varphi_v \times  N /| \varphi_v \times N | = \varphi_v\times N/|\varphi_v|$
as $| \varphi_v \times N | = |\varphi_v| $. Then 
\ben\label{eq:equation_boudnary}
(X\circ \varphi)\times \varphi_v \cdot {\bf N} = (X\cdot \nu)\circ \varphi |\varphi_v|.
\een
So inserting \eqref{eq:equation_boudnary} and \eqref{eq:X_times_phi} into \eqref{eq:per_surf} we obtain \eqref{eq:struc_surf}.
\end{proof}

\begin{rem}
Formula \eqref{eq:struc_surf}  may be rewritten as
\ben
\begin{split}
dJ(M)(X) & = \int_M \mathcal H  (X\cdot \mathfrak n) \;  ds +  \int_{\partial M}  X\cdot \nu   \, ds,
\end{split}
\een
where $\mathfrak n$ and $\mathcal H$ are the unit normal field and mean curvature  on $M$, 
respectively. 
So by definition $n\circ \varphi = N$ and $\mathcal H \circ \varphi = H$. 
Also in this case our main theorem is satisfied and we recover 
\eqref{eq:struc_manifold} with 
$$ h(X^\bot) = \int_M \mathcal H  (X^\bot \cdot \mathfrak n) \;  ds, \qquad g(X_{|\partial M} \cdot \nu ) =  \int_{\partial M}  X_{|\partial M} \cdot \nu   \, ds.  $$
\end{rem}

\subsection{A shape gradient of order one}
Provided that the manifold $M$ is smooth enough we have 
seen in the examples from the previous sections that the 
shape derivative was always a distribution of 
order zero in the sense that $g$ and $h$ were 
linear functionals on $C^0(\partial M)$ respectively 
$C^0(M)$.

Let $\gamma:[0,L] \rightarrow \R^2$ an arc-length parametrised regular curve. 
Then $M:=\gamma([0,L])$ is a closed submanifold of $\R^2$. We define the 
{\bf elastic energy} associated with $\gamma$ as
$$ E(M) := \int_0^L \kappa^2 \, ds, $$
where $\kappa$ denotes the curvature of $\gamma$. 
Here we are interested in the unconstrained case, where we do not impose
any further conditions at the end points of the curve. 

Let us introduce some notation.   We define the tangent vector field along $\gamma$ by $T:= \gamma'$ and 
$N:= RT$ where $R$ denotes the 90 degrees counter-clockwise rotation matrix in $\R^2$.  
Further we denote by $\mathfrak n:M\rightarrow \R^2$ the unit normal field and by 
$\mathfrak t:M\rightarrow \R^2$ the tangent field 'living' on 
$M$ so that by definition  $N = \mathfrak n \circ \gamma$ and $T = \mathfrak t\circ \gamma $. 
Note that by definition $T' = \kappa N$ and $N' = - \kappa T$. 

For the derivation of the first variation of the anisotropic elastic  energy 
with fixed end points, we refer the reader 
to \cite[Lem. 2.2, p. 502]{garcke_barrett_nuernberg2012}.

\begin{lem}
 Let  $\gamma:[0,L] \rightarrow \R^2$ be a $C^2$-regular arc-length 
 parametrised embedded curve such that $M:= \gamma([0,L]) \subset D$. For 
 every $X\in C^2_c(D,\R^d)$,  we have
\begin{align*}
dE(M;X)  =&  \int_0^L   (2\kappa''  + \kappa^3) (X  \cdot \mathfrak n)\circ\gamma  \, ds
+ \bigg[   2\kappa \nabla_\Gamma( X\cdot \mathfrak n )\circ \gamma \cdot 
\gamma'  -  2\kappa' (X  \cdot \mathfrak n) \circ \gamma \bigg]^L_0.
\end{align*}
\end{lem}
\begin{proof} Denote by $\Phi_t$ the flow generated by $X\in C^2_c(D,\R^2)$. Set $\gamma_t(s) := \Phi_t(\gamma(s))$. We compute
\begin{equation}\label{dE_free_elastic}
\begin{split}
 dE(M;X) & = \dt\left( \int_0^L \kappa_t^2 |\gamma_t'|\, ds\right)_{|t=0} \\
& =  \int_0^L 2\kappa \dot \kappa + \dot \gamma'\cdot T \kappa^2 \, ds.  
\end{split}
\end{equation} 
Let us determine a formula for the variation of the curvature, that is, $\dot \kappa$. Differentiating
$ T_t' = \kappa_t v_t N_t$ we obtain $\dot T' = \dot\kappa N + \kappa \dot N + \dot v \kappa N$ and thus
 \ben\label{eq:dot_Tb}
  \dot T'\cdot N = \dot \kappa + \dot v \kappa 
 \een
and differentiating $\gamma'_t =  v_t T_t$ yields  
$  \dot \gamma_t' =  \dot v_t T_t    + v_t \dot T_t $ from whence we get by another 
differentiation 
$\dot \gamma'' =    \dot v' T + \dot v \kappa N  + \dot T',    $
where we used $(v_t)'_{|t=0} =0 $. 
So 
\ben\label{eq:dot_Tb-2}
 \dot T'\cdot N = \dot \gamma''\cdot N - \dot v \kappa.
 \een
Putting  \eqref{eq:dot_Tb} and \eqref{eq:dot_Tb-2} together, we obtain
\ben\label{dot_kappa}
 \dot \kappa =  \dot \gamma''\cdot N  - 2 \dot v \kappa.
\een
So plugging \eqref{dot_kappa} into \eqref{dE_free_elastic}
and integrating by parts, we obtain 
\begin{align*} dE(M;X)  =&  \int_0^L 2\kappa (  \dot \gamma''\cdot N  - 2 \dot v \kappa ) + \dot \gamma'\cdot T \kappa^2 \, ds\\
  =&  \int_0^L 2\kappa  \dot \gamma''\cdot N  - 3 \kappa^2 \dot \gamma'\cdot T \, ds\\
  =&  \int_0^L  - 2\kappa'  \dot \gamma'\cdot N \underbrace{- 2\kappa  \dot \gamma'\cdot N' - 3 \kappa^2 \dot \gamma'\cdot T}_{- \kappa^2 \dot \gamma'\cdot T} \, ds 
  + \bigg[_0^L 2\kappa  \dot \gamma'\cdot N\bigg]^T_0 \\
  =& \bigg[ 2\kappa  \dot \gamma'\cdot N\bigg]_0^L  -  \int_0^L   2\kappa'  \dot \gamma'\cdot N + \kappa^2 \dot \gamma'\cdot T \, ds\\
   =&  \int_0^L   (2\kappa''  + \kappa^3)(X\circ \gamma) \cdot N  \, ds
    + \bigg[2\kappa  \dot \gamma'\cdot N - 2\kappa' \dot\gamma\cdot N - \kappa^2  \dot\gamma \cdot T \bigg]^T_0.
    \end{align*}
On account of the identities 
$ 2\kappa (\dot \gamma \cdot N)' = 2\kappa \dot \gamma' \cdot N - 2 \kappa^2 \dot \gamma\cdot T $
and  $\dot \gamma(s) = X\circ \gamma(s)$ and 
$(\dot \gamma \cdot N)'   = \nabla_\Gamma(X\cdot \mathfrak n)\circ \gamma \cdot \gamma' $, we recover the desired 
formula. 
\end{proof}

\begin{rem}
We see that also in this case  \eqref{eq:struc_manifold} is satisfied. We have
$$ h(X^\bot) =  \int_0^L   (2\kappa''  + \kappa^3) (X^\bot  \cdot \mathfrak n)\circ\gamma  \, ds
- \bigg[  2\kappa' (X^\bot  \cdot \mathfrak n) \circ \gamma \bigg]^L_0 + 
\bigg[   2\kappa \nabla_\Gamma( X^\bot\cdot \mathfrak n )\circ \gamma \cdot 
\gamma' \bigg]^L_0 
$$
and  $g=0$. Note that  $h:C^1( M) \rightarrow \R$ 
is a  distribution of order $k=1$. This well-known result is interesting 
as it gives an example for which $g=0$ although the 
manifold $M$ has non-empty boundary $\partial M\ne \emptyset$; compare Corollary \ref{cor:structure1}.  
Note that if we fixed the end points of $\gamma$, then the term 
$\big[  2\kappa' (X  \cdot \mathfrak n) \circ \gamma \big]^L_0 = 0$ because 
$X(\gamma(0)) = X(\gamma(L)) = 0$. 
\end{rem}

\end{document}